\documentclass[12pt]{article}%
\usepackage{amsmath,amssymb,amsfonts,amscd,amsthm}
\usepackage{amsfonts}
\usepackage{amsmath}
\usepackage{graphicx}%

\usepackage{tikz}
\usetikzlibrary{arrows,shapes,positioning}
\usetikzlibrary{decorations.markings}
\tikzstyle arrowstyle=[scale=1]
\tikzstyle directed=[postaction={decorate,decoration={markings,
    mark=at position .65 with {\arrow[arrowstyle]{stealth}}}}]
\tikzstyle reverse directed=[postaction={decorate,decoration={markings,
    mark=at position .65 with {\arrowreversed[arrowstyle]{stealth};}}}]
    
\providecommand{\U}[1]{\protect\rule{.1in}{.1in}}
\newtheorem{theorem}{Theorem}

\newtheorem{cor}[theorem]{Corollary}

\newtheorem{definition}[theorem]{Definition}

\newtheorem{lemma}[theorem]{Lemma}

\newtheorem{prop}[theorem]{Proposition}
\newtheorem{rem}[theorem]{Remark}

\DeclareMathOperator{\sign}{sgn}

\newcommand {\R} {\mathbb{R}}

\newcommand {\C} {\mathbb{C}}

\newcommand {\q} {\mathbf{q}}
\newcommand {\p} {\mathbf{p}}

\newcommand {\diag} {\text{diag}}

\newcommand {\J} {\mathbf{J}}

\begin{document}
  \begin{center}
        {\fontsize{18}{22}\selectfont
       \bf Stability of fixed points and associated relative equilibria of the $3$-body problem on $\mathbb S^1$ and $\mathbb S^2$}
       \end{center}

\vspace{4mm}

        \begin{center}
        {\bf Florin Diacu$^{1,2}$, Juan Manuel S\'anchez-Cerritos$^3$, and Shuqiang Zhu$^2$}\\
\bigskip
$^1$Pacific Institute for the Mathematical Sciences\\
and\\
$^2$Department of Mathematics and Statistics\\ University of Victoria, Victoria, Canada,\\
\bigskip
$^3$Departamento de Matem\'aticas\\
Universidad Aut\'onoma Metropolitana - Iztapalapa, Mexico, D.F., Mexico\\ 
\bigskip
diacu@uvic.ca, jmsc@xanum.uam.mx,  zhus@uvic.ca  
       \end{center}

        
        \begin{center}
        \today
        \end{center}

        
\abstract{We prove that the fixed points of the curved 3-body problem and their associated relative equilibria are Lyapunov stable if the solutions are  restricted to $\mathbb S^1$, but their linear stability depends on the angular velocity  if the bodies are considered on $\mathbb S^2$. More precisely, the associated relative equilibria are linearly stable if and only if the angular velocity is greater than a certain critical value.}\\
        
        
\section{Introduction}
The  curved 3-body problem is a natural extension of the classical 3-body problem to spaces of constant nonzero Gaussian curvature. Its history, which started with Bolyai and Lobachevsky, is outlined in \cite{diacu1}, where it is also shown that, in the 2-dimensional case, the study of the problem can be reduced to the unit sphere, $\mathbb S^2$, and the unit hyperbolic sphere, $\mathbb H^2$. The equations of motion can be written as a Hamiltonian system in extrinsic coordinates with holonomic constraints in the Euclidean space, for positive curvature, but in the Minkowski space, for negative curvature. This formulation led to fruitful results, especially in the study of symmetric motions, \cite{diacu1, diacu2, diacu4, diacu3, diacu5, diacu10, diacu12, diacu14, martinez, zhu1}. 

One class of solutions are the fixed points, which occur only on spheres, \cite{diacu1, diacu2}. We study here their stability as well as that of their associated relative equilibria for the 3-body problem on $\mathbb S^1$ and $\mathbb S^2$. Fixed points are critical points of the force function that defines the system on the configuration space. For the bodies $m_1, m_2, m_3$, they give rise to periodic relative equilibria, which in spherical coordinates $(\varphi,\theta)$ have the form 
\[\theta_i(t)=\theta_i, \ \ \varphi_i(t)=\varphi_i +\omega t, \ \ \dot{\theta}_i=0,\  \ \dot{\varphi}_i=\omega,\ \ i\in\{1,2,3\}, \]
for any $\omega\ne 0$. There have been some previous studies of the stability of orbits in the curved 3- and 4-body problem, \cite{diacu12, martinez}, but none of them considered fixed points,
so this appears to be a first attempt in this direction. 

For fixed points of the 3-body problem on $\mathbb S^2$, the masses must lie on a great circle of $\mathbb S^2$ at the vertices of an acute triangle, \cite{diacu1}. In Section 2, we prove a general criterion for the existence of fixed points for $n>2$ masses. It is known that not all masses can form fixed points, \cite{diacu3}, and in Section 3 we find all mass triples that do so. In Section 4, we investigate the stability of fixed points and of their associated relative equilibria on $\mathbb S^1$. Since the associated relative equilibria are periodic orbits, their stability is defined as the stability of the corresponding rest points of the flow in the reduced phase space, \cite{simo2}. To perform the reduction, we use coordinates similar to the Jacobi coordinates of the classical $n$-body problem. By  showing that the corresponding rest points are maxima of the force function, we prove that the fixed points and the associated relative equilibria are Lyapunov stable on $\mathbb S^1$. In Section 5, we study the linear stability of  these solutions on $\mathbb S^2$. Their linear stability is defined as the stability of the corresponding rest points of the linearized system restricted to  proper linear subspaces. Using an idea from the work of Rick Moeckel, \cite{moeckel1}, we determine proper linear subspaces to study the fixed-point solutions and the relative equilibria. Then we show that their linear stability depends on the angular velocity $\omega$. More precisely, the associated relative equilibria are linearly stable if and only if the angular velocity is greater than a certain value determined by the configuration.

\section{Existence of fixed points on $\mathbb S^1$}
The goal of this section is to introduce the curved $n$-body problem on $\mathbb S^2$, define a class of orbits we call fixed points, and provide a criterion for the existence of these solutions for any number $n>2$ of point masses.  

A Hamiltonian system is given by the triple $(H, M, w)$, where $M$ is an even dimensional manifold, $w$ is a symplectic form on $M$, and $H$ is an infinitely differentiable, real-valued function on $M$.  The curved $n$-body problem on $\mathbb S^2, n\ge 2,$ is given by a Hamiltonian system $(H, M, w)$ as follows. Let $m_1,\dots,m_n$ be a system of $n$ positive point masses (bodies), whose configuration is given by the vector 
$$
\q=(\q_1, \cdots, \q_n)\in (\mathbb S^2)^n,
$$ 
with $\q_i=(x_i,y_i,z_i)=(\sin \theta_i \cos \varphi_i, \sin \theta_i \sin \varphi_i,\cos \theta_i)$ in spherical coordinates $(\varphi_i, \theta_i)$, $i\in\{1,\dots,n\}$. We denote by $d_{ij}$ the geodesic distance between the point masses $m_i$ and $m_j, \ i,j\in\{1,\dots,n\}$. The kinetic energy has the form 
\[
T\left(\varphi_1, \theta_1, \cdots, \varphi_n, \theta_n\right):=\sum_{i=1}^n \frac{m_i}{2}\left(\dot{\theta}_i^2 +\sin^2\theta_i\dot{\varphi}_i^2\right),
\] 
and the conjugate momenta are given by 
\[ 
p_{\theta_i}=m_i\dot{\theta}_i, \ \  p_{\varphi_i}=m_i\sin^2\theta_i\dot{\varphi}_i, \ i\in\{1,\dots,n\}.  
\]
The configuration space is $W=(\mathbb S^2)^n\setminus\Delta$, where $\Delta$ is the singularity set, 
\[\Delta:=\{\q \in (\mathbb S^2)^n: \exists (i,j) \ \mbox{with $i\ne j$ such that} \ d_{ij}=0 \ \mbox{or}\   d_{ij}= \pi\}. \]
The cotangent potential on $W$, which extends the Newtonian potential to the sphere, is given by
\[ 
U\left(\varphi_1, \theta_1, \cdots, \varphi_n, \theta_n\right):=-\sum_{1\le i< j\le n} m_im_j\cot d_{ij}. 
\]
Let $d$ denote the differential, $V=-U$ be the force function, and $T^*W$ represent the cotangent bundle of $W$. Then the curved $n$-body problem on $\mathbb S^2$ is described by the Hamiltonian system 
\begin{equation}\label{Ham}
H=T+U=T-V, \ \ M= T^*W, \ \ w=d\left(\sum_i  p_{\theta_i} d\theta_i +p_{\varphi_i} d \varphi_i\right). \end{equation}

\begin{definition}
A configuration $\q\in W=(\mathbb S^2)^n\setminus\Delta$ is called a fixed point if it is a critical point of the force function $V$. 
\end{definition}

The terminology is justified since, once placed at such an initial configuration in $\mathbb S^2$ with zero initial velocities, the bodies don't move, so we have a rest point of the system. Therefore a configuration of this kind can be also thought as a solution of \eqref{Ham} with zero momenta. 

Since in this paper we are mainly interested in fixed points of the 3-body problem, we must note that the three bodies always lie on a geodesic $\mathbb S^1$, as shown in \cite{diacu1}. So for our further purposes we will assume that all $n$ bodies lie on $\mathbb S^1$. Without loss of generality, we can take $\mathbb S^1$ to be the equator, so a fixed-point solution has the form
\begin{equation} 
\theta_i(t)=\frac{\pi}{2}, \ \ \varphi_i(t)=\varphi_i\ ({\rm constant}), \ \ p_{\theta_i}= p_{\varphi_i}=0, \ \ i\in\{1,\dots,n\},
\label{fpsolution}
\end{equation}
If we rotate a fixed point uniformly about the $z$-axis with any angular velocity $\omega\ne 0$, we obtain an associated  relative equilibrium:
\begin{equation} 
\theta_i(t)=\frac{\pi}{2}, \ \ \varphi_i(t)=a_i +\omega t, \ \ p_{\theta_i}=0,\ \  p_{\varphi_i}=m_i\omega,\ \ i\in\{1,\dots,n\},
\label{resolution}
\end{equation} 
where $a_1,\dots, a_n$ are constants.

It has been known that  fixed points cannot lie within one hemisphere, \cite{diacu1}. We briefly sketch the proof here. We use the Cartesian coordinates of $\mathbb S^2$, $\q_i=(x_i,y_i,z_i)$. Using the fact that 
$x_i^2+y_i^2+z_i^2=1$ and $\cos d_{ij}=\q_i \cdot \q_j$, direct computation shows $\frac{\partial d_{ij}}{\partial z_i}=\frac {z_j-\cos d_{ij}z_i}{-\sin d_{ij}}$. 
Suppose that $\q$ is a configuration with all particles  on the north hemisphere and $q_1$ is the lowest, i.e., $z_i\ge z_1\ge 0$ for all $i$, and that not all  particles are on the equator. Then 
\[ \frac{\partial d_{1j}}{\partial z_1}=\frac {z_j-\cos d_{1j}z_1}{-\sin d_{1j}}\le\frac {z_1(1-\cos d_{1j})}{-\sin d_{1j}}<0.  \]   
Then $\frac{\partial V}{\partial z_1}=\sum_{j=1,j\ne i}^n\frac{-m_im_j}{\sin^2 d_{1j}}\frac{\partial d_{1j}}{\partial z_1}>0$. Hence $\q$ could not be a critical point of $V$.  
 
Let us now provide a general criterion for the existence of fixed points on $\mathbb S^1$ for any number $n>2$  masses. Because of the $SO(2)$-symmetry that acts on the bodies, we can assume without loss of generality that 
$$
0=\varphi_1< \varphi_2<\cdots<\varphi_n<2\pi, \ 
$$ and 
$
\varphi_{i+1}-\varphi_{i}<\pi,\ \ i\in\{1,\dots,n-1\}\ \ {\rm and}\ \  \varphi_{n}-\varphi_{1}>\pi.
$

 \begin{prop}
 A configuration $\q$ on $(\mathbb S^1)^n\setminus\Delta$ is a fixed point  if and only if 
 \begin{equation} 
  \sum_{i=1, i\neq k}^N \frac{m_km_i\sin (\varphi_k-\varphi_i) }{\sin ^3 d_{ki}}=0,
 \label{nfp}
 \end{equation}
 for each $k \in \{1,\dots,n\}$.
 \end{prop}
\begin{proof}
Fixed points are critical points of the force function $V=\sum_{1\le i<j\le n} m_im_j \cot d_{ij}.$
Notice that  $d_{ki} \in (0,\pi)$ could be either $|\varphi_k-\varphi_i|$ or  $2\pi-|\varphi_k-\varphi_i|$.  Then $\frac{\partial d_{ki}}{\partial \varphi_k}$ takes the values $\pm1$. These facts are summarized in the following table.
\begin{center}
\begin{tabular}{|p{2cm}|p{4cm}|p{1cm}|}
  \hline
  $\varphi_k-\varphi_i$ & $d_{ki}$ &  $\frac{\partial d_{ki}}{\partial \varphi_k}$  \\ \hline
   $(-2\pi, -\pi)$ &$2\pi + \varphi_k-\varphi_i$ & 1 \\ 
  \hline
  $(-\pi, 0)$ &$ \varphi_i-\varphi_k$ & -1\\ 
    \hline
    $(0, \pi)$ &$\varphi_k-\varphi_i$ & 1   \\ 
      \hline
      $(\pi, 2\pi)$& $2\pi - \varphi_k+\varphi_i$ & -1   \\ 
        \hline
\end{tabular}
\end{center}
Then
$$
\frac{\partial d_{ki}}{\partial \varphi_k} = \sign( \sin (\varphi_k-\varphi_i)) = \frac{ \sin (\varphi_k-\varphi_i)}{ |\sin (\varphi_k-\varphi_i)|} =\frac{ \sin (\varphi_k-\varphi_i)}{ \sin d_{ki}}.
$$
Taking  the partial derivatives of $V$ respect to  $\varphi_k,\ k\in\{1,\dots,n\}$, we obtain that
$$
0= \sum_{i=1, i\neq k}^n \frac{m_km_i }{\sin ^2 d_{ki}} \frac{\partial d_{ki}}{\partial \varphi_k}= \sum_{i=1, i\neq k}^n \frac{m_km_i\sin (\varphi_k-\varphi_i) }{\sin ^3 d_{ki}}
$$
for each $k \in \{1,\cdots,n\}$, a remark that completes the proof.
\end{proof}

\section{Triangular fixed points}\label{configuration}

In this section we focus on the case $n=3$. We set $\alpha=\varphi_2-\varphi_1, \ \beta=\varphi_3-\varphi_2$. Then, since $ \varphi_{i+1}-\varphi_{i}<\pi$, we obtain
\begin{equation} 
  0<\alpha<\pi,\ \ \ 0<\beta<\pi,\ \ \ \pi<\alpha+\beta<2\pi
  \label{acutecondition}
  \end{equation} 
and 
\[d_{12}=\alpha,\ d_{23}=\beta, \ d_{13}=2\pi-(\alpha+\beta).\] 
From the point of view of Euclidean geometry, a fixed point for three masses forms a triangle in the $xy$-plane. We have $\angle \q_2\q_1\q_3=\frac{d_{23}}{2}<\frac{\pi}{2}$, and, similarly, the two other angles are acute. Thus each fixed point must be an acute triangle (see Figure \ref{fig:acute}).

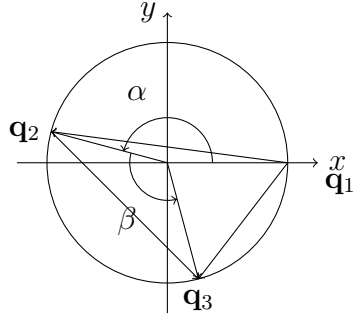
\begin{figure}[!h]
	\centering
  \begin{tikzpicture}
  \draw (0,  0) circle (1.6);
  \draw[->] (0,0) -- (165:1.6);
  \draw[->] (0,0) -- (285:1.6);
  \draw[->] (0:1.6) -- (165:1.6);
  \draw[->] (0:1.6) -- (285:1.6);
  \draw[->] (165:1.6) -- (285:1.6);
  \draw[->] (0:.6)arc(0:165:.6);
  \draw[->] (165:.5)arc(165:285:.5);
  
  \node[above] at(120:.8) {$\alpha$};
  \node[below] at(220:.7) {$\beta$};
  \node[below] at(0:2.3) {$\mathbf{q}_1$};
  \node[left] at (165:1.6) {$\mathbf{q}_2$};
  \node[below] at (285:1.6) {$\mathbf{q}_3$};
  
   \draw [->] (0, -2)--(0, 2) node (zaxis) [left] {$y$};
         \draw [->](-2, 0)--(2, 0) node (yaxis) [right] {$x$};
  \end{tikzpicture}
  \caption{An acute triangle fixed-point configuration}
  \label{fig:acute}
  \end{figure}

Equations (\ref{nfp}) above can be written in the form  
  \begin{equation} 
   \frac{m_2}{\sin^2 \alpha}= \frac{m_3}{\sin^2 (\alpha +\beta)},\ \ \ \ \frac{m_1}{\sin^2 \alpha}= \frac{m_3}{\sin^2 \beta},\ \ \ \ 
   \frac{m_2}{\sin^2 \beta}= \frac{m_1}{\sin^2 (\alpha +\beta)}.
   \label{3fp}
   \end{equation}
With the help of these equations, we can now find all fixed points. It is known on one hand that, for any acute triangle configuration, there are mass triples $(m_1,m_2,m_3)$ that generate  fixed points, \cite{diacu3, zhu1}. On the other hand,  not all mass triples can form fixed points. So for what mass triples, are there fixed points and how many such fixed points occur in each case?  We count these configurations by fixing 
$$
\varphi_1=0<\varphi_2<\varphi_3<2\pi,
$$ 
since any other configurations differs from this one by an isometry. It is easy to see that for any acute triangle there are infinitely many admissible mass triples,  
\begin{equation} 
(m_1,m_2,m_3)=m_3\left(\frac{\sin^2 \alpha}{\sin^2 \beta}, \frac{\sin^2 \alpha}{\sin^2 (\alpha+\beta)}, 1\right). 
\label{3fp1}
\end{equation}

Hence we normalize the masses by taking $m_1+m_2+m_3=1$.  Then the admissible masses can be represented by $(m_1,m_2)\in (0,1)\times(0,1)$. Let us denote by $\mathcal{U}$ the open region in $\R^2$ defined by (\ref{acutecondition}), i.e., $0<\alpha<\pi, 0<\beta<\pi, \pi<\alpha+\beta<2\pi.$
Then (\ref{3fp1}) defines a mapping,
\begin{equation} 
\begin{split}
\chi: \mathcal{U} &\to (0,1)\times(0,1),\\
      (\alpha,\beta)& \mapsto (m_1,m_2)=\frac{1}{\tau}\left(\frac{\sin^2 \alpha}{\sin^2 \beta},\frac{\sin^2 \alpha}{\sin^2 (\alpha+\beta)} \right),
\end{split}
\notag
\end{equation}
where $\tau=\frac{\sin^2 \alpha}{\sin^2 \beta}+ \frac{\sin^2 \alpha}{\sin^2 (\alpha+\beta)}+ 1$. 

Then we can formulate the question related to the admissible masses in more precise terms: What is the image of $\chi$? Is $\chi$ injective? Or, in other words, for what pairs $(m_1,m_2)$ can we find an acute triangle fixed-point configuration? Is such configuration unique?  The following result provides the answers to these questions.
 \begin{theorem}\label{image}
 Let  
 \[ \mathcal{V}:=\{ (m_1,m_2) \in (0,1)\times(0,1):\ m_1^2m_2^2+ m_1^2m_3^2+m_2^2m_3^2 -2m_1m_2m_3<0\},\] 
where $m_3=1-(m_1+m_2)$  (see Figure 2 ). 
 Then there exists an acute triangle fixed-point configuration only for mass triples  $(m_1,m_2,m_3=1-m_1-m_2)$ such that  $(m_1,m_2)\in \mathcal{V}$. Furthermore, such a configuration is unique. 
 \end{theorem} 
 
\begin{proof}
Our approach is to find the inverse mapping $(\alpha , \beta )= \chi^{-1}(m_1,m_2)$. Then inequalities (\ref{acutecondition}) would yield inequalities for $(m_1,m_2)$, which would give $\mathcal{V}:=\chi(\mathcal{U})$, and since the inverse mapping exists, $\chi: \mathcal{U} \to \mathcal{V}$ must be bijective, so the configuration is unique. 

By (\ref{3fp}), we set  
$ u=\sqrt{\frac{m_3}{m_1}}=\frac{\sin\beta}{\sin \alpha}, \  v=\sqrt{\frac{m_3}{m_2}}=-\frac{\sin(\alpha+\beta)}{\sin \alpha}.$ 
 Then 
\[ -v=\cos\beta + \frac{\sin\beta}{\sin \alpha}\cos \alpha= \pm \sqrt{1-u^2\sin^2 \alpha} + u\cos \alpha,  \]
\[ (-u\cos \alpha-v )^2= u^2\cos^2 \alpha+2uv\cos\alpha+v^2 =1-u^2\sin^2 \alpha,
   \]
and we obtain 
\begin{equation} 
\begin{split}
\cos \alpha &= \frac{1-u^2-v^2}{2uv}= \frac{m_1m_2-m_3(m_1+m_2)}{2m_3\sqrt{m_1m_2}}\\ 
\cos \beta &= -v -u\cos \alpha =\frac{u^2-v^2-1}{2v}=\frac{m_3(m_2-m_1)-m_1m_2}{2m_1\sqrt{m_2m_3}}.
\end{split}
\label{inverse}
\end{equation}
 Both $\cos \alpha$ and $\cos \beta$ are injective on $\mathcal{U}$, so the inverse mapping   $\chi^{-1}$ is  
 \[\alpha=\cos^{-1}\frac{m_1m_2-m_3(m_1+m_2)}{2m_3\sqrt{m_1m_2}}, \ \  \beta=\cos^{-1}\frac{m_3(m_2-m_1)-m_1m_2}{2m_1\sqrt{m_2m_3}}.  \]
 To find $\chi(\mathcal{U})$, note that the condition $(\ref{acutecondition})$ is equivalent to 
 \[\cos^2 \alpha <1, \ \ \cos^2 \beta <1,\ \  \sin(\alpha+\beta)<0. \]
 Substituting (\ref{inverse}) into the above inequalities and using the normalization condition $m_1+m_2+m_3=1$, we find that the three inequalities lead to the same relationship:
 \begin{equation} 
 \begin{split}  m_1^2m_2^2+m_2^2m_3^2+m_1^2m_3^2-2m_1m_2m_3<0,\notag
 \end{split}
 \end{equation}
 which is the image of the mapping $\chi$ and it gives all admissible mass triples, a remark that completes the proof (see also Figure 2).
 \end{proof}
 \begin{figure}
   \begin{center}
   \includegraphics [width=0.35 \textwidth] {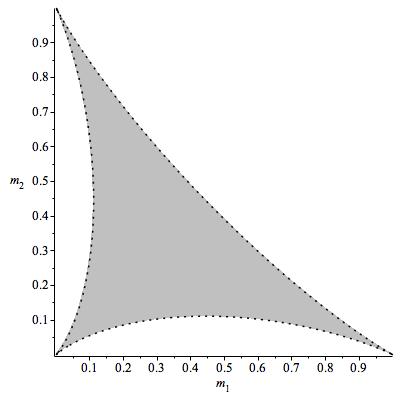}
   \end{center}
   \caption{The region $\mathcal{V}$}
   \end{figure}\label{fig:region}
   
We end this section with a property of the acute isosceles triangle fixed-point configurations. The result below was first obtained in \cite{diacu3}, so here we come up with a simpler proof. 

\begin{cor}
For each acute isosceles triangle fixed-point configuration, if the masses $m_1$ and $m_3$ are at the vertices of the base, then $m_1=m_3$ and $m_1<4m_2$.  Reciprocally,  if  $m_1=m_3$ and $m_1<4m_2$, then the only kind of fixed point these three masses can form is an acute isosceles triangle.
\end{cor}

\begin{proof}
Let the fixed-point configuration be an acute isosceles triangle, for instance, $\alpha=\beta$, then by $(\ref{3fp1})$, we have $m_1=m_3$. Consequently
\begin{equation} 
\begin{split}
 m_1^2m_2^2+m_2^2m_3^2+m_1^2m_3^2-2m_1m_2m_3=m_1^2\left(m_1^2+2m_2^2-2m_2\right)<0.
\end{split}
\notag
\end{equation}
Therefore
\[m_1^2<2m_2(1-m_2)=2m_2\cdot2m_1,\]
which implies that
\[m_1<4m_2.\]
Reciprocally, if  $m_1=m_3$ and $m_1<4m_2$, then the bodies can form an acute isosceles triangle fixed-point configuration by the first part of this corollary. They can only form such configuration by Theorem \ref{image}.
\end{proof}

\section{Reduction and stability on $\mathbb S^1$}

It is easy to see that if the initial velocity of each particle is confined to the tangent space $T(\mathbb S^1)$, then the motion takes place in $T^*(\mathbb S^1)^3$ forever, where $T^*(\mathbb S^1)$ is the cotangent bundle of $\mathbb S^1$. In other words, $T^*(\mathbb S^1)^3$ is an invariant manifold of the equations of motion. In this section, we are going to study the stability of all acute triangle fixed-point solutions and their associated relative equilibria on this invariant manifold. We will first perform the reduction of the 3-body problem to $T^*(\mathbb S^1)^3$ and then prove that all these solutions are Lyapunov stable.

Confined to $T^*(\mathbb S^1)^3$, the Hamiltonian system take the form
\[ H(\varphi_i,  p_{\varphi_i} )=\sum_i \frac{1}{2}\frac{p^2_{\varphi_i}}{m_i}-\sum_{j\neq i} m_im_j\cot d_{ij},\ T^*(\mathbb S^1)^3,  \ w=d\left(\sum_i p_{\varphi_i} d \varphi_i\right). \]
 It is easy to see that in these coordinates, the conserved total angular momentum \cite{diacu1} is
 given by the expression 
\[ \J(\varphi_i,  p_{\varphi_i})=p_{\varphi_1}+p_{\varphi_2}+p_{\varphi_3}.\]

We further restrict the study of the stability of  relative equilibria to the quotient space $S'_{p}$ of $T^*(\mathbb S^1)^3$, which we define using
\[ S_{p}:=\{  (\q,\p) \in T^*(\mathbb S^1)^3:\J =p   \}, \]
and let $S'_{p}$ be the quotient space under the Lie group $SO(2)$. For all values of $p$, these spaces are always smooth manifolds and have dimension $4$.  We can apply this procedure to the fixed-point solutions as well. The advantage of this approach is to eliminate the drift caused by the $SO(2)$ symmetry on the bodies. Indeed, by symmetry, if we perturb a fixed-point solution with some initial velocity $\dot{\varphi}_i=\epsilon$, then we obtain a relative equilibrium with angular velocity $\epsilon$.  We can therefore introduce the definition below.

\begin{definition}
A fixed-point solution (or its associated relative equilibrium) is Lyapunov stable  on $\mathbb S^1$ if the corresponding rest point  of the reduced Hamiltonian system on the quotient space $S'_p$ is Lyapunov stable.
\end{definition}

Thus we need to explicitly find the quotient manifold $S'_p$ and the \emph{reduced Hamiltonian} $H_{p}$ (see \cite{singer} for a theoretical approach to this kind of problem). The angular momentum here behaves like the linear momentum of the classical $n$-body problem. Thus to perform the reduction, we introduce a type of Jacobi coordinates \cite{MH92}:
\[ \tilde{\varphi}=\frac{1}{\bar{m}} (m_1\varphi_1 + m_2\varphi_2 + m_3\varphi_3 ),\ \ \ \phi_1=\varphi_2-\varphi_1, \ \ \ \phi_2=\varphi_3-\nu_1\varphi_1-\nu_2\varphi_2,\]
where
\[ \bar{m}=m_1+m_2+m_3, \ \ \ \nu_1= \frac{m_1}{m_1+m_2}, \ \ \ \nu_2= \frac{m_2}{m_1+m_2}.\]
The corresponding conjugate momenta are then given by
\[ p_{\tilde{\varphi}}=\bar{m}\dot{\tilde{\varphi}},\ \ \ p_{\phi_1}=\nu_3\dot{\phi}_1, \ \ \  p_{\phi_2}=\nu_4\dot{\phi}_2, \]   
where $ \nu_3=\frac{m_1m_2}{m_1+m_2}$ and  $\nu_4=\frac{(m_1+m_2)m_3}{\bar{m}}.$  It is easy to verify that 
\[ p_{\tilde{\varphi}}d\tilde{\varphi}+p_{\phi_1}d\phi_1+ p_{\phi_2}d\phi_2=p_{\varphi_1} d \varphi_1+p_{\varphi_2} d\varphi_2+p_{\varphi_3} d \varphi_3\] 
 and that the Hamiltonian is 
 \[ H= \frac{1}{2}\left( \frac{p^2_{\tilde{\varphi}}}{\bar{m}}+ \frac{p^2_{\phi_1}}{\nu_3}+\frac{p^2_{\phi_2}}{\nu_4}\right) -V(\phi_1,\phi_2), \]
 where $V(\phi_1,\phi_2)$ is the force function in the new variables.  Note that $\varphi_2-\varphi_1=\phi_1=\alpha$, and 
 \[\beta= \varphi_3-\varphi_2=\phi_2+ \nu_1\varphi_1+\nu_2\varphi_2-\varphi_2=\phi_2-\nu_1\phi_1.   \]
Recall that 
$ d_{12}=\alpha, d_{23}=\beta, d_{13}=2\pi -(\alpha+\beta).$ Thus
\begin{equation*}
\begin{split}
V(\phi_1,\phi_2)&= m_1m_2\cot \alpha +m_2m_3\cot \beta -m_1m_3\cot (\alpha +\beta)\\
                &=m_1m_2\cot \phi_1 +m_2m_3\cot (\phi_2-\nu_1\phi_1) -m_1m_3\cot (\phi_2+\nu_2\phi_1).
\end{split}
\end{equation*} 

  We now study the stability of the fixed-point solution. In this case $\dot{\varphi}_1=\dot{\varphi}_2=\dot{\varphi}_3=0,$ 
   we have $\bar{m}\dot{\tilde{\varphi}}=0$, and $\J =0$. So the quotient manifold is $S'_0$. On this quotient  manifold, $p_{\tilde{\varphi}}=0$ and  we can set  $\tilde{\varphi}=0$ since  we identify all points that differ by a rotation.   Thus  we  use  $\phi_1,\phi_2, p_{\phi_1}, p_{\phi_2}$ as the coordinates of the 4-dimensional manifold $S'_0$. Under these coordinates, we have the reduced Hamiltonian function
  \[ H_0= \frac{1}{2}\left( \frac{p^2_{\phi_1}}{\nu_3}+\frac{p^2_{\phi_2}}{\nu_4}\right) -V(\phi_1,\phi_2).\] 
  Suppose $\q$ is a fixed-point configuration on $\mathbb S^1$ of three masses $m_1,m_2,m_3$. Their positions are given by 
  \[ \varphi_1=0, \ \varphi_2=\alpha,\ \ \varphi_3=\alpha+\beta,  \] 
  (see Figure \ref{fig:acute}), where $\alpha, \beta$ are in region \eqref{acutecondition}. By Section \ref{configuration}, the masses and the two angles are related  by equations \eqref{3fp}. 
  Then the fixed-point solution $(0,\alpha,\alpha+\beta, 0,0,0)$ we want to study in $T^*(\mathbb S^1)^3$ becomes a rest point of the reduced Hamiltonian system  on $S'_0$: 
  \[\phi_1=\alpha, \ \ \ \phi_2= \beta + \nu_1 \alpha, \ \ \   p_{\phi_1}=0, \   \ \ p_{\phi_2}=0. \]
  Let us denote this rest point by $X$. Our goal is to study the stability of $X$, a property given by the following result.

\begin{theorem}\label{stableons1}
Every  acute triangle fixed-point  solution is Lyapunov  stable on $\mathbb S^1$. 
\end{theorem}
\begin{proof}
We will show that the rest point $X$  is a local minimum of the Hamiltonian $H_0$. Since the Hamiltonian is preserved during any motion, we can then conclude that the rest point $X$ is stable.  The rest point $X$ is obviously a minimum of the kinetic energy. We will show that it is also a local maximum of $V$ by studying its Hessian matrix. 

We use the variables $(\alpha,\beta)$ instead of $(\phi_1,\phi_2)$  . Then  
\[ V(\alpha,\beta)=m_1m_2\cot \alpha +m_2m_3\cot \beta -m_1m_3\cot (\alpha +\beta),\] and the Hessian matrix $H(V)$ of $V$ at $X$ has the form     
\[ H(V)= 2\begin{bmatrix}\frac{m_1m_2 \cos \alpha}{\sin^3 \alpha}-\frac{m_1m_3 \cos (\alpha +\beta)}{\sin^3 (\alpha +\beta)}&  -\frac{m_1m_3 \cos (\alpha +\beta)}{\sin^3 (\alpha +\beta)}\\
 -\frac{m_1m_3 \cos (\alpha +\beta)}{\sin^3 (\alpha +\beta)}& \frac{m_2m_3 \cos \beta}{\sin^3 \beta}-\frac{m_1m_3 \cos (\alpha +\beta)}{\sin^3 (\alpha +\beta)}
  \end{bmatrix}.  \]
Recall that equations (\ref{3fp})  are 
 \[ \frac{m_2}{\sin^2 \alpha}= \frac{m_3}{\sin^2 (\alpha +\beta)},\ \ 
 \frac{m_2}{\sin^2 \beta}= \frac{m_1}{\sin^2 (\alpha +\beta)},\ \ 
 \frac{m_1}{\sin^2 \alpha}= \frac{m_3}{\sin^2 \beta}.\]
By direct computation,  we find that the diagonal entries of $\frac{1}{2}H(V)$ are
 \begin{align*}
   \frac{m_1m_2 \cos \alpha}{\sin^3 \alpha}-\frac{m_1m_3 \cos (\alpha +\beta)}{\sin^3 (\alpha +\beta)} &= \frac{m_1m_2}{\sin^2 \alpha}\frac{\sin \beta } {\sin (\alpha +\beta)\sin \alpha} &<0,\\
   \frac{m_2m_3 \cos \beta}{\sin^3 \beta}-\frac{m_1m_3 \cos (\alpha +\beta)}{\sin^3 (\alpha +\beta)} &= \frac{m_1m_3}{\sin^2 (\alpha +\beta)}\frac{\sin \alpha } {\sin (\alpha +\beta)\sin \beta} &<0.
 \end{align*}
Hence the trace of $H(V)$ is negative.  
 Here we use the fact that $(\alpha,\beta)$ is in region (\ref{acutecondition}), thus 
$ \sin\alpha >0,$ $ \sin\beta >0,$ and $ \sin(\alpha+\beta) <0.$  
Further straightforward computations show that the determinant of $\frac{1}{2}H(V)$ is
 \begin{equation*}
 \begin{split}
 & \frac{m_1m_2 \cos \alpha}{\sin^3 \alpha}\frac{m_2m_3 \cos \beta}{\sin^3 \beta}-   \left( \frac{m_1m_2 \cos \alpha}{\sin^3 \alpha}+\frac{m_2m_3 \cos \beta}{\sin^3 \beta}\right)\frac{m_1m_3 \cos (\alpha +\beta)}{\sin^3 (\alpha +\beta)}\\
 &=\frac{m_1m_3m_2^2 \cos \alpha \cos \beta}{\sin^3 \alpha\sin^3 \beta}-\frac{m_1m_3}{\sin^2 (\alpha +\beta)}\frac{\sin (\alpha +\beta)} {\sin \alpha \sin \beta}   \frac{m_1m_3 \cos (\alpha +\beta)}{\sin^3 (\alpha +\beta)}\\
 &= \frac{m_1m_3m_2^2 }{\sin^3 \alpha\sin^3 \beta}\left( \cos \alpha \cos \beta-\cos (\alpha +\beta)    \right)= \frac{m_1m_3m_2^2 }{\sin^2 \alpha\sin^2 \beta}>0.
 \end{split}
 \end{equation*}
These two facts imply that the two eigenvalues of $H(V)$ are both negative. Then the fixed-point configuration is a local maximum of $V$. We can thus conclude that the rest point $X=( \alpha,  \beta + \nu_1 \alpha,  0,0 )$ is a local minimum of $H=T-V$. This remark completes the proof. 
\end{proof}
 
We further study the stability of the associated relative equilibria. In this case $\dot{\varphi}_1=\dot{\varphi}_2=\dot{\varphi}_3=\omega,$ we have  $\bar{m}\dot{\tilde{\varphi}}=\bar{m}\omega$, and  $\J=\bar{m}\omega$. So the quotient manifold is $S'_{\bar{m}\omega}$. On this quotient manifold, $p_{\tilde{\varphi}}=\bar{m}\omega$ and  we can set  $\tilde{\varphi}=0$ since  we identify all points that differ by a rotation.
 Thus   we  use  $\phi_1,\phi_2, p_{\phi_1}, p_{\phi_2}$ as the coordinates of the 4-dimensional manifold $S'_{\bar{m}\omega}$. Under these coordinates, we have the reduced Hamiltonian
 \[ H_{\bar{m}\omega}= \frac{1}{2}\left( \frac{p^2_{\phi_1}}{\nu_3}+\frac{p^2_{\phi_2}}{\nu_4}\right) -V(\phi_1,\phi_2) + \frac{\bar{m}\omega^2}{2}.\]

Then the relative equilibrium (\ref{resolution}) in $T^*(\mathbb S^1)^3$ becomes a rest point  in $S'_{\bar{m}\omega}$, 
\[\phi_1=\alpha, \ \phi_2= \beta + \nu_1 \alpha, \  p_{\phi_1}=0, \   p_{\phi_2}=0. \]
Let us denote this rest point by $X_1$. Note that $H_0$ and $H_{\bar{m}\omega}$ are the same up to a constant. Then we can conclude that $X_1$ is also a local minimum of $H_{\bar{m}\omega}$. So we have proved the following result.

\begin{theorem}
Every relative equilibrium  associated to an acute triangle  fixed-point configuration is Lyapunov stable on $\mathbb S^1$. 
\end{theorem}
 
\begin{rem}
The method of proving the stability of relative equilibria by showing that the rest points are minima of the Hamiltonian never works in the classical $n$-body problem. The obstacle lies in  the fact that planar central configurations of the classical $n$-body problem  are never  maxima of the force function, \cite{moeckel1}. 
\end{rem}

\section{Stability on $\mathbb S^2$}
In this section we study the linear stability of the above solutions on $\mathbb S^2$. Different from the previous case, their stability depends on the angular velocity $\omega$. We first introduce rotating coordinates to treat a general relative equilibrium on $\mathbb S^2$ as rest point, and obtain the linearized system 
$\dot{v}=L(v-X_\omega)$.
  We then compute $L$ for relative equilibria associated with  fixed-point configurations on the equator. As in the classical $n$-body problem, we study the stability of the rest points on a proper subspace. Inspired by the work of Moeckel \cite{moeckel1}, we find the proper linear subspace. In the end, we show that these solutions are linearly stable if $\omega^2$ is greater than a critical value.

So consider a  general relative equilibrium on $\mathbb S^2$ with angular velocity $\omega$ and introduce the rotating coordinates
\[\overline{\theta}_i= \theta_i,  \ \  \dot{\overline{\theta}}_i= \dot{\theta}_i, \ \  \overline{\varphi}_i= \varphi_i-\omega t, \ \  \dot{\overline{\varphi}}_i= \dot{\varphi}_i-\omega, \ \  \overline{p}_{\theta_i}=p_{\theta_i}, \ \  \overline{p}_{\varphi_i}=p_{\varphi_i},
i\in\{1,2,\dots,n\}.\] 
In these new coordinates, the original Hamiltonian system  
\begin{align}
& \dot{\theta}_i= \frac{\partial H}{\partial p_{\theta_i} }=\frac{p_{\theta_i}}{m_i}, 
 &\dot{{p}}_{\theta_i}&=-\frac{\partial H}{\partial \theta_i }=\frac{   p^2_{\varphi_i}\cos\theta_i}{m_i \sin^3\theta_i} - \frac{\partial U}{\partial \theta_i},\notag \\
&\dot{\varphi}_i=\frac{\partial H}{\partial p_{\varphi_i} }=\frac{p_{\varphi_i}}{m_i\sin^2\theta_i}, 
&\dot{p}_{\varphi_i}&=-\frac{\partial H}{\partial \varphi_i}=-\frac{\partial U}{\partial \varphi_i},\ i\in\{1,2,\dots,n\},\notag
\end{align}
becomes 
\begin{align}
&\dot{\overline{\theta}}_i=\frac{\overline{p}_{\theta_i}}{m_i},  
&\dot{\overline{p}}_{\theta_i}&=\frac{   \overline{p}^2_{\varphi_i}\cos\overline{\theta}_i}{m_i \sin^3\overline{\theta}_i} - \frac{\partial U}{\partial \overline{\theta}_i}, \notag \\ 
&\dot{\overline{\varphi}}_i=\frac{\overline{p}_{\varphi_i}}{m_i\sin^2\overline{\theta}_i}-\omega, 
&\dot{\overline{p}}_{\varphi_i}&=- \frac{\partial U}{\partial \overline{\varphi}_i},\ i\in\{1,2,\dots,n\}. \label{hamiltonequation}
\end{align}
Thus system (\ref{hamiltonequation}) is Hamiltonian with  
\[ H= \sum_{i=1}^n \frac{1}{2}\left(\frac{\overline{p}^2_{\theta_i}}{m_i} + \frac{\overline{p}^2_{\varphi_i}}{m_i\sin^2\overline{\theta}_i} -2\overline{p}_{\varphi_i} \omega\right) +U \]
and symplectic form ${\displaystyle w=d\left(\sum_{i=1}^n  \overline{p}_{\theta_i} d\overline{\theta}_i +\overline{p}_{\varphi_i} d \overline{\varphi}_i\right) }$. 

We will use $\theta_i,\varphi_i, p_{\theta_i},p_{\varphi_i}$ instead of $\overline{\theta}_i,\overline{\varphi}_i, \overline{p}_{\theta_i} \overline{p}_{\varphi_i}$ if no further confusion arises.  Denote by $X_{\omega}$ the rest point in system (\ref{hamiltonequation}) corresponding to a  relative equilibrium (\ref{resolution}) with angular velocity $\omega$. Then $X_\omega$ is 
\[ \theta_i(t)=\theta_i(0), \ \ \varphi_i(t)=\varphi_i(0), \ \ p_{\theta_i}=0, \ \ p_{\varphi_i}=\omega m_i\sin^2\theta_i, \]
and we are going  to study the stability of $X_\omega$ for the linearized system
\begin{equation} \dot{v}=L(v-X_\omega), \ \ v=(\theta_1,..,\theta_n,\varphi_1,..,\varphi_n,  p_{\theta_1},..,p_{\theta_n}, p_{\varphi_1}, ..,  p_{\varphi_n}   ).  
\label{lsystem}\end{equation}

By straightforward computation, we get 
\begin{frame}
\footnotesize
\arraycolsep=3pt 
\medmuskip = 1mu 
\[L= \begin{bmatrix}0&0&M^{-1}&0\\
K&0&0&M^{-1}C^{-1}\\
\frac{\partial^2 V}{\partial \theta_i\partial\theta_j}+ N& \frac{\partial^2 V}{\partial \theta_i\partial\varphi_j}&0&-K^T\\
 \frac{\partial^2 V}{\partial \varphi_i\partial\theta_j}& \frac{\partial^2 V}{\partial \varphi_i\partial\varphi_j} &0&0\end{bmatrix},\ \  
 N= 
 \begin{bmatrix}
  \frac{-p^2_{\varphi_1} (1+2\cos^2\theta_1)}{m_1\sin^4\theta_1}&\cdots&0\\
  0&\frac{-p^2_{\varphi_2} (1+2\cos^2\theta_2)}{m_2\sin^4\theta_2}&0\\
  \vdots & \cdots & \vdots\\
  0&\cdots&\frac{-p^2_{\varphi_n} (1+2\cos^2\theta_n)}{m_n\sin^4\theta_n}\\
   \end{bmatrix},\]
 
\[
K= 
 \begin{bmatrix}
  \frac{-2p_{\varphi_1} \cos\theta_1}{m_1\sin^3\theta_1}&\cdots&0\\
  0&\frac{-2p_{\varphi_2} \cos\theta_2}{m_2\sin^3\theta_2}&0\\
  \vdots & \cdots & \vdots\\
  0&\cdots&\frac{-2p_{\varphi_n} \cos\theta_n}{m_n\sin^3\theta_n}\\
   \end{bmatrix}, \ M^{-1}= 
\begin{bmatrix}
\frac{1}{m_1}&\cdots&0\\
 0&\frac{1}{m_2}&0\\
  \vdots & \cdots & \vdots\\
 0&\cdots&\frac{1}{m_n}\\
  \end{bmatrix}, \ \  
 C^{-1}=
 \begin{bmatrix}
 \frac{1}{\sin^2\theta_1}&\cdots&0\\
  0&\frac{1}{\sin^2\theta_2}&0\\
   \vdots & \cdots & \vdots\\
  0&\cdots&\frac{1}{\sin^2\theta_n}\\
   \end{bmatrix}, \ \ 
  \]
    \end{frame}
where $L$ is a $4n\times 4n$ matrix, whereas $N, K, M^{-1}$, and $C^{-1}$ are $n\times n$ matrices.
  
It is generally difficult to find the normal form of $L$. However, for  relative equilibria   associated to  a fixed-point configuration on the equator, things are easier. 

\begin{lemma}\label{Jacobian}
 For fixed-point configurations on the equator, $L= \begin{bmatrix}0&0&M^{-1}&0\\
 0&0&0&M^{-1}\\
 H_\omega& 0&0&0\\
  0& G &0&0\end{bmatrix},$ where $H_\omega=H-\omega^2 M=\left[\frac{\partial^2 V}{\partial \theta_i\partial\theta_j}\right]-\omega^2 M$ and $G= \left[\frac{\partial^2 V}{\partial \varphi_i\partial\varphi_j} \right]$. And the elements of $H$ and $G$ are 
    \[H_{ij}=\frac{m_im_j}{\sin^3d_{ij}}, \ \  H_{ii}=-\sum_{j\ne i,j=1}^{n}H_{ij}\cos d_{ij}, \ \ G_{ij}= \frac{-2m_im_j\cos d_{ij} }{\sin^3d_{ij}}, \ \  G_{ii}=-\sum_{j\ne i,j=1}^{n}G_{ij}.\]
\end{lemma}

\begin{proof}
 In this case $\theta_i=\frac{\pi}{2}$, so  $K=0$, $C^{-1}=I_n$. Since that $d_{ij}=cos^{-1}\q_i\cdot\q_j$, we have  
 \[ \frac{\partial d_{ij}}{\partial \theta_i}= \frac{-1}{\sin d_{ij}} \left(\q_j\cdot\frac{\partial\q_i}{\partial \theta_i}\right), \  \  \frac{\partial d_{ij}}{\partial \varphi_i}=\frac{-1}{\sin d_{ij}} \left(\q_j\cdot\frac{\partial\q_i}{\partial \varphi_i}\right),  \  \]
 and 
 \begin{equation} 
 \begin{split}
 \frac{\partial^2 V}{\partial \theta_i\partial\theta_j}&=m_im_j\left[ \frac{3\q_i\cdot\q_j}{\sin^5d_{ij}}\left(\q_j\cdot\frac{\partial\q_i}{\partial \theta_i}\right)\left(\q_i\cdot\frac{\partial\q_j}{\partial \theta_j}\right)+\frac{1}{\sin^3d_{ij}}\frac{\partial\q_j}{\partial \theta_j}\cdot\frac{\partial\q_i}{\partial \theta_i}\right],\\
\frac{\partial^2 V}{\partial \theta_i\partial\theta_i}&=\sum_{j\ne i,j=1}^{n}
 m_im_j\left[\frac{3\q_i\cdot\q_j}{\sin^5d_{ij}}\left(\q_j\cdot\frac{\partial\q_i}{\partial \theta_i}\right)^2+\frac{1}{\sin^3d_{ij}}\q_j\cdot\frac{\partial^2\q_i}{\partial \theta_i^2}\right],\\
\frac{\partial^2 V}{\partial \varphi_i\partial\varphi_j}&
 =m_im_j\left[\frac{3\q_i\cdot\q_j}{\sin^5d_{ij}}\left(\q_j\cdot\frac{\partial\q_i}{\partial \varphi_i}\right)\left(\q_i\cdot\frac{\partial\q_j}{\partial \varphi_j}\right)+\frac{1}{\sin^3d_{ij}}\frac{\partial\q_j}{\partial \varphi_j}\cdot\frac{\partial\q_i}{\partial \varphi_i}\right],\\
 \frac{\partial^2 V}{\partial \varphi_i\partial\varphi_i}&=
 \sum_{j\ne i,j=1}^{n}
 m_im_j\left[\frac{3\q_i\cdot\q_j}{\sin^5d_{ij}}\left(\q_j\cdot\frac{\partial\q_i}{\partial \varphi_i}\right)^2+\frac{1}{\sin^3d_{ij}}\q_j\cdot\frac{\partial^2\q_i}{\partial \varphi_i^2}\right],\\
 \frac{\partial^2 V}{\partial \theta_i\partial\varphi_j}&
 =m_im_j\left[\frac{3\q_i\cdot\q_j}{\sin^5d_{ij}}\left(\q_j\cdot\frac{\partial\q_i}{\partial \theta_i}\right)\left(\q_i\cdot\frac{\partial\q_j}{\partial \varphi_j}\right)+\frac{1}{\sin^3d_{ij}}\frac{\partial\q_j}{\partial \varphi_j}\cdot\frac{\partial\q_i}{\partial \theta_i}\right],\\
 \frac{\partial^2 V}{\partial \theta_i\partial\varphi_i}&=\sum_{j\ne i,j=1}^{n}m_im_j\left[\frac{3\q_i\cdot\q_j}{\sin^5d_{ij}}\left(\q_j\cdot\frac{\partial\q_i}{\partial \varphi_i}\right)\left(\q_j\cdot\frac{\partial\q_i}{\partial \theta_i}\right)+\frac{1}{\sin^3d_{ij}}\q_j\cdot\frac{\partial^2\q_i}{\partial \varphi_i\partial\theta_i}\right].\notag
 \end{split}
 \end{equation}
Note that  \[\frac{\partial\q_i}{\partial \theta_i}=\frac{\partial}{\partial \theta_i} (\sin\theta_i\cos\varphi_i, \sin\theta_i\sin\varphi_i, \cos\theta_i) =(\cos\theta_i\cos\varphi_i, \cos\theta_i\sin\varphi_i, -\sin\theta_i)=(0, 0, -1).\] 
Similarly we obtain
\begin{align*} 
&\frac{\partial\q_i}{\partial \varphi_i}=(-\sin\varphi_i,\cos\varphi_i,  0), &\frac{\partial^2\q_i}{\partial \theta_i\partial \varphi_i}&=(0,0, 0),\\
&\frac{\partial^2\q_i}{\partial \theta^2_i}=-(\cos\varphi_i, \sin\varphi_i, 0),
&\frac{\partial^2\q_i}{\partial \varphi^2_i}&=-(\cos\varphi_i, \sin\varphi_i, 0).
\notag
\end{align*}
Thus we have 
\begin{align*}
 &\q_j\cdot\frac{\partial\q_i}{\partial \theta_i}=0,  &\frac{\partial\q_j}{\partial \theta_j}\cdot\frac{\partial\q_i}{\partial \theta_i}=1, \ \   &\q_j\cdot\frac{\partial^2\q_i}{\partial \theta^2_i}= -\cos d_{ij}, 
& \q_j\cdot\frac{\partial\q_i}{\partial \varphi_i}=\pm \sin d_{ij}, \\ 
&\frac{\partial\q_j}{\partial \varphi_j}\cdot\frac{\partial\q_i}{\partial \theta_i}=0, \ \ &\q_j\cdot\frac{\partial^2\q_i}{\partial \theta_i\partial \varphi_i}= 0, \ \ \
&  \frac{\partial\q_j}{\partial \varphi_j}\cdot\frac{\partial\q_i}{\partial \varphi_i}=\cos d_{ij},  \ \  &\q_j\cdot\frac{\partial^2\q_i}{\partial \varphi^2_i}= -\cos d_{ij}. 
  \notag
\end{align*}
Then straightforward computation shows that the block $\left[\frac{\partial^2 V}{\partial \theta_i\partial\varphi_j} \right]$  is zero,  and 
\[H_{ij}=\frac{m_im_j}{\sin^3d_{ij}}, \ \  H_{ii}=-\sum_{j\ne i,j=1}^{n}H_{ij}\cos d_{ij}, \ \ G_{ij}= \frac{-2m_im_j\cos d_{ij} }{\sin^3d_{ij}}, \ \  G_{ii}=-\sum_{j\ne i,j=1}^{n}G_{ij},\]
 a remark that completes the proof.
\end{proof}

We can now find the normal form of $L$ by computing the normal forms of $H_\omega M^{-1}$ and $GM^{-1}$.
\begin{lemma} \label{eigenspace}
$H_\omega M^{-1}$ and $GM^{-1}$ are diagonalizable. If $u\in \C^n$ is an eigenvector of $H_\omega M^{-1}$ ($GM^{-1}$) with eigenvalue $\lambda \neq 0$, then there exist a two dimensional invariant subspace of $L$ in $\C^{4n}$  on which  $L$ is 
$\begin{bmatrix}
 \sqrt{\lambda}&0\\0&-\sqrt{\lambda}
 \end{bmatrix}$.
If $u\in \C^n$ is an eigenvector of $H_\omega M^{-1}$ ($GM^{-1}$) with eigenvalue  $0$, then there exist a two dimensional  invariant subspace of $L$ in $\C^{4n}$  on which  $L$ is   
$\begin{bmatrix}
 0&1\\0&0
 \end{bmatrix}$.
\end{lemma}

\begin{proof}
It is enough to prove this for $H_\omega M^{-1}$. Note $H_\omega $ is symmetric, thus 
$H_\omega M^{-1}$ is symmetric with respect to the inner product $\langle v,w\rangle=v^TM^{-1}w$ so it is diagonalizable with respect to some $M^{-1}$ orthogonal basis.  Now suppose $ H_\omega M^{-1} u = \lambda u$, $\lambda \neq 0$. Then 
\begin{frame}
\footnotesize
\arraycolsep=3pt 
\medmuskip = 1mu 
\[ \begin{bmatrix}0&0&M^{-1}&0\\
 0&0&0&M^{-1}\\
 H_\omega &0&0&0\\
  0&G &0&0\end{bmatrix}\begin{bmatrix}\frac{M^{-1}u}{\sqrt{\lambda}}&\frac{M^{-1}u}{- \sqrt{\lambda}}\\0&0\\u&u\\0&0
  \end{bmatrix}=\begin{bmatrix}M^{-1}u&M^{-1}u\\0&0\\ \sqrt{\lambda}u&-\sqrt{\lambda}u\\0&0
    \end{bmatrix} =\begin{bmatrix}\frac{M^{-1}u}{\sqrt{\lambda}}&\frac{M^{-1}u}{- \sqrt{\lambda}}\\0&0\\u&u\\0&0
    \end{bmatrix} \begin{bmatrix}\sqrt{\lambda}&0\\0&-\sqrt{\lambda}  \end{bmatrix}. \]
\end{frame}
Similarly, if $H_\omega M^{-1}u =0$, then 
\[   \begin{bmatrix}0&0&M^{-1}&0\\
 0&0&0&M^{-1}\\
 H_\omega & 0&0&0\\
  0& G &0&0\end{bmatrix}\begin{bmatrix}M^{-1}u&0\\0&0\\0&u\\0&0
  \end{bmatrix}=\begin{bmatrix}0&M^{-1}u\\0&0\\0&0\\0&0
    \end{bmatrix} =\begin{bmatrix}M^{-1}u&0\\0&0\\0&u\\0&0
      \end{bmatrix}\begin{bmatrix}
       0&1\\0&0
       \end{bmatrix}.\]
This completes the proof.
\end{proof}

Normally, a rest point $X_\omega$ is called linearly stable if $X_\omega$ is a stable rest point of the linearized system (\ref{lsystem}). However, as for relative equilibria of the classical $n$-body problem, the symmetries and integrals of the problem make it impossible to satisfy this condition.

Note that $V$ is invariant under the $SO(3)$ action, which implies that any fixed-point configuration remains a fixed-point configuration after any rotation. We can find the three vectors  corresponding to the three rotations:
\begin{equation}v_1=\begin{bmatrix}
\cos d_{11}\\ \vdots\\ \cos d_{1n}
\end{bmatrix} =\begin{bmatrix}
x_1\\ \vdots\\ x_n
\end{bmatrix},   \ \   v_2=\begin{bmatrix}
\sin (\varphi_1-\varphi_1)\\\vdots\\\sin (\varphi_n-\varphi_1)
\end{bmatrix} =\begin{bmatrix}
y_1\\\vdots\\y_n
\end{bmatrix},   \ \   v_3= \begin{bmatrix}
1\\\vdots\\1
\end{bmatrix}.  \label{nullvectors}
\end{equation}
\begin{lemma}\label{eigenvalue}
Let $H$ and $G$ be the matrices defined in Lemma \ref{Jacobian} and $v_1, v_2, v_3$ the vectors defined in (\ref{nullvectors}). Then
\[Hv_1=0, \ \ Hv_2=0, \ \ Gv_3=0.  \]
\end{lemma}
\begin{proof}
Note that
\[\cos d_{1i} =\cos (\varphi_1-\varphi_k+\varphi_k-\varphi_i)=\cos d_{1k} \cos d_{ki} -\sin (\varphi_1-\varphi_k) \sin(\varphi_k-\varphi_i),\]
 \[\sin (\varphi_i-\varphi_1) = \sin (\varphi_i-\varphi_k+\varphi_k-\varphi_1)        
            =\sin (\varphi_k-\varphi_1)\cos d_{ki} -\cos d_{1k} \sin(\varphi_k-\varphi_i).\]
Then the  $k$-th entry of $Hv_1$ is 
\begin{equation} 
\begin{split}
 & H_{k1}\cos d_{11}+ \cdots+ \left(-\sum _{i\neq k}H_{ki}\cos d_{ki}\right)\cos d_{1k}+ \cdots+H_{kn}\cos d_{1n} \\
&=\frac{m_km_1\cos d_{11}}{\sin^3d_{k1}}+\cdots-\left(\sum _{i\neq k}\frac{m_km_i\cos d_{1k}\cos d_{ki} }{\sin^3d_{ki}}\right)+ \cdots+ \frac{m_km_n\cos d_{1n}}{\sin^3d_{kn}} \\
&=\sum _{i\neq k} m_km_i\frac{\cos d_{1i}-\cos d_{1k}\cos d_{ki} }{\sin^3d_{ki}}=-\sin (\varphi_1-\varphi_k) \sum _{i\neq k} m_km_i\frac{\sin(\varphi_k-\varphi_i)\ }{\sin^3d_{ki}}.\\
\end{split}
\notag
\end{equation}
And the  $k$-th entry of $Hv_2$ is 
\begin{equation} 
\begin{split}
 & H_{k1}\sin (\varphi_1-\varphi_1)+ \cdots+ \left(-\sum _{i\neq k}H_{ki} \cos d_{ki}\right)\sin (\varphi_k-\varphi_1)+ \cdots+ H_{kn} \sin (\varphi_n-\varphi_1)\\
&=\sum _{i\neq k} m_km_i\frac{\sin (\varphi_i-\varphi_1)-\sin (\varphi_k-\varphi_1)\cos d_{ki} }{\sin^3d_{ki}}=-\cos d_{k1}  \sum _{i\neq k} m_km_i\frac{\sin(\varphi_k-\varphi_i)\ }{\sin^3d_{ki}}.\notag
\end{split}
\end{equation}
Using criterion (\ref{nfp}) of fixed-point configurations for $n$ masses, 
$$
\sum_{i, i\neq k} \frac{m_km_i\sin (\varphi_k-\varphi_i) }{\sin ^3 d_{ki}}=0,
$$ 
we find that $Hv_1= Hv_2=0$. Note that $\sum_{i=1}^n G_{ik}=0$ for all $k$. Hence $Gv_3=0$, a remark that completes the proof.
\end{proof}

Now we consider the stability of the fixed-point solutions, i.e, $X_0$. In this case, $H_\omega=H$. 
Consider the 6-dimensional subspace $E_1$ of $\C^{4n}$ spanned by the vectors
\[ \begin{bmatrix}v_1\\0\\0\\0
\end{bmatrix}, \   \begin{bmatrix}0\\0\\Mv_1\\0
\end{bmatrix}, \  \begin{bmatrix}v_2\\0\\0\\0
\end{bmatrix}, \   \begin{bmatrix}0\\0\\Mv_2\\0
\end{bmatrix}, \   \begin{bmatrix}0\\v_3\\0\\0
\end{bmatrix}, \   \begin{bmatrix}0\\0\\0\\Mv_3
\end{bmatrix}.    \]
Then using Lemma \ref{eigenvalue}, we obtain that $E_1$ is an invariant subspace for $L$. The matrix of $L|_{E_1}$ in this basis is
\[L|_{E_1}=\diag \Bigg\{\begin{bmatrix}
 0&1\\0&0
 \end{bmatrix}, \begin{bmatrix}
  0&1\\0&0
  \end{bmatrix},\begin{bmatrix}
   0&1\\0&0
   \end{bmatrix} \Bigg \}.   \]
Though all eigenvalues on $E_1$ are $0$, there are three nontrivial Jordan blocks, a fact which implies that the rest point is not linearly stable in the conventional sense. This instability is trivial as a natural effect of the symmetry of this Hamiltonian system. Indeed,
we can perturb a fixed-point solution into a relative equilibrium by any rotation in $SO(3)$. Then the angular positions of these orbits drift away from each other, a property mathematically reflected by the nontrivial Jordan blocks, as remarked in \cite{moeckel3}.  

It is traditional in celestial mechanics to view the drifts in this subspace as harmless. Indeed, they can be eliminated  by  fixing the angular momentum and passing to a quotient manifold under the action of the rotational symmetry group, \cite{moeckel3}. 
 Thus it is reasonable to formulate a definition of linear stability based on the behaviour of $L$ in a complementary subspace, \cite{moeckel1}. To define such a subspace, it is necessary to introduce the \emph{skew inner product} of two complex vectors $v,w \in \C^{4n}:$
\[ \Omega(v,w)=v^TJw, \ \ J_{4n \times 4n}= \begin{bmatrix} 0&-I\\
I&0\end{bmatrix}. \]
Using the fact that $L=JS$, $S^T=S$,  where $S$ is the Hessian matrix of the Hamiltonian at $X_0$, we obtain that
 \[ \Omega(v,Lw)=-\Omega(Lv,w).\]
With the help of this property it is easy to show that the skew-orthogonal complement of an invariant subspace of $L$ is again invariant. Indeed, let $E$ denote the skew orthogonal complement in $\C^{4n}$ of $E_1$, that is,
\[E=\{ v\in \C^{4n}: \Omega(v,w)=0 \ \mbox{for all}\ w\in E_1\}.\]
 Then $E$ is an $L$ invariant subspace of dimension $4n-6$.
 \begin{definition}
  A fixed-point solutions $X_0$ associated with a fixed-point configuration on the equator is called \emph{linearly stable} if $X_0$ is a stable rest point of the restriction of the linearized equation (\ref{lsystem}) to $E$.
  \end{definition}

 For $n=3,$ i.e., the relative equilibria associated with acute triangle fixed-point configurations, we have $E=\C^6$, and we can find the  normal form of $L|_E$. 

\begin{theorem} \label{main}
For each acute triangle fixed-point solution, $L|_E$ is diagonalizable and in properly chosen basis,
\[  L|_E=\diag \Big\{ \sqrt{\lambda_1}, -\sqrt{\lambda_1},   i\sqrt{|\lambda_2|}, -i\sqrt{|\lambda_2|},   i\sqrt{|\lambda_3|},-i\sqrt{|\lambda_3|}\Big\},\]  
where  $0,0,\lambda_1>0$ are  the eigenvalues of $HM^{-1}$, and $0,\lambda_2<0,\lambda_3<0$ are  the eigenvalues of $GM^{-1}$. All  acute triangle fixed-point solutions on the equator  are unstable on $\mathbb S^2$.  
\end{theorem}

\begin{proof} 
We first find the eigenvalues of $HM^{-1}$ and $GM^{-1}$. Recall that 
$ d_{12}=\alpha, d_{23}=\beta, d_{13}=2\pi -(\alpha+\beta).$
Using $\alpha, \beta$, by direct computation, we obtain that
\[ 
HM^{-1}=
   \begin{bmatrix}\frac{-m_2\cos \alpha}{\sin^3\alpha}+\frac{m_3\cos (\alpha+\beta)}{\sin^3(\alpha+\beta)}&\frac{m_1}{\sin^3\alpha}&\frac{-m_1}{\sin^3(\alpha+\beta)}\\
   \frac{m_2}{\sin^3\alpha}&\frac{-m_1\cos \alpha}{\sin^3\alpha}+\frac{-m_3\cos \beta}{\sin^3\beta}&\frac{m_2}{\sin^3\beta}\\
   \frac{-m_3}{\sin^3(\alpha+\beta)}&\frac{m_3}{\sin^3\beta}&\frac{m_1\cos (\alpha+\beta)}{\sin^3(\alpha+\beta)}+\frac{-m_2\cos \beta}{\sin^3\beta}\\  
     \end{bmatrix}. \ \]  
Lemma \ref{eigenvalue} implies that  $Mv_1$ and $Mv_2$ are two eigenvectors of $HM^{-1}$ with eigenvalue $0$. Thus the other eigenvalue $\lambda_1$ equals the trace of the matrix. Using the same idea as in the proof of Theorem  \ref{stableons1}, we find that the second diagonal entry is
\[   \frac{-m_1\cos \alpha}{\sin^3\alpha}+\frac{-m_3\cos \beta}{\sin^3\beta}= -\frac{m_3}{\sin^2\beta}\frac{\sin(\alpha+\beta)}{\sin\alpha\sin\beta}>0.\]
The first one and the third one are just opposite to diagonal entries of the matrix in the proof of Theorem \ref{stableons1}, so they are positive. Hence $\lambda_1$ is positive since it equals  the trace. 

Lemma \ref{eigenvalue} also implies that $G$ has one eigenvalue 0. Note that the proof of Theorem \ref{stableons1} implies that  fixed-point configurations are local maxima of $V$ on $(\mathbb S^1)^3\setminus\Delta$. Thus the two other eigenvalues of $G=\left[ \frac{\partial^2 V}{\partial \varphi_i\partial\varphi_j}\right]$ are both negative. Note that $M^{-\frac{1}{2}}$ is well defined. Then  $G$ is congruent to $ G_1:=(M^{-\frac{1}{2}})^T G M^{-\frac{1}{2}}$, which is similar to $M^{\frac{1}{2}} G_1  M^{-\frac{1}{2}}=GM^{-1}$.   By Sylvester's law of inertia, \cite{Fuh12},  we have 
\[n_-(GM^{-1} )=n_-(G)=2, \ \  n_0(GM^{-1} )=n_0(G)=1,\]
where $n_0(A)$ is the number of zero eigenvalues and $n_-(A)$ is the number of negative eigenvalues of matrix $A$.
This proves the eigenvalues of $HM^{-1}$ are $0,0,\lambda_1>0$, and the eigenvalues of $GM^{-1}$ are $0,\lambda_2<0,\lambda_3<0$. 

 By Lemma \ref{eigenspace},  we see that  $L$ on $\C^{12}$ is similar to 
 \[ \diag \Bigg \{\begin{bmatrix}
        0&1\\0&0
        \end{bmatrix},\begin{bmatrix}
               0&1\\0&0
               \end{bmatrix},\begin{bmatrix}
                      0&1\\0&0
                      \end{bmatrix},  \sqrt{\lambda_1}, -\sqrt{\lambda_1},   i\sqrt{|\lambda_2|}, -i\sqrt{|\lambda_2|},   i\sqrt{|\lambda_3|},-i\sqrt{|\lambda_3|} \Bigg \}. \]
 Recall that the  normal form of $L|_{E_1}$ is given by the first three nontrivial Jordan blocks.  We thus obtain that $L|_E$ is similar to 
  \[ \diag\Big\{ \sqrt{\lambda_1}, -\sqrt{\lambda_1},   i\sqrt{|\lambda_2|}, -i\sqrt{|\lambda_2|},   i\sqrt{|\lambda_3|},-i\sqrt{|\lambda_3|}\Big \}.\] 
  Then the positive eigenvalue $\sqrt{\lambda_1}$ indicates that  all  acute triangle fixed-point solutions and all associated relative equilibria are unstable on $\mathbb S^2$, a remark that completes the proof.  
\end{proof}

Now we study the linear stability of the associated relative equilibria. Define $E_2$ as the two-dimensional subspace spanned by 
\[  [0,v_3^T,0,0]^T, \ \ [  0,0,0,(Mv_3)^T]^T. \]
Then $L$ has one Jordan block on $E_2$. By the same reason, it is reasonable to define stability based on the behaviour of the system on the complementary space, that is
\[\tilde{E}=\{ v\in \C^{4n}: \Omega(v,w)=0 \ \mbox{for all}\ w\in E_2\}.\]
 Then $\tilde{E}$ is an $L$ invariant subspace of dimension $4n-2$.    
\begin{definition}
 A relative equilibrium $X_\omega$ associated with a fixed-point configuration on the equator is called \emph{linearly stable} if $X_\omega$ is a stable rest point of the restriction of the linearized equation (\ref{lsystem}) to $\tilde{E}$.
 \end{definition} 
 
 \begin{theorem} \label{main-1}
 Let $X_\omega$ be a relative equilibria   associated with an  acute triangle fixed-point configuration  on the equator. Then it is unstable on $\mathbb S^2$ if and only if $0< \omega^2 \le \lambda_1 $, 
 and it is linearly stable if and only  $ \lambda_1  <\omega^2$, where 
 \[  \lambda_1 = -\frac{m_2}{\sin^2 \alpha}\frac{\sin \beta } {\sin (\alpha +\beta)\sin \alpha} -\frac{m_3}{\sin^2 \beta}\frac{\sin \alpha } {\sin (\alpha +\beta)\sin \beta}-\frac{m_3}{\sin^2\beta}\frac{\sin(\alpha+\beta)}{\sin\alpha\sin\beta}.    \]  
 \end{theorem}
 
 \begin{proof}
 By Lemma \ref{eigenspace} and the proof of the above theorem, we only need to find the eigenvalues of $HM^{-1} -\omega^2$. Note that $HM^{-1}$ has three eigenvectors, with two eigenvalues being zero. Thus there exists an invertible matrix $P$ such that 
 $$
 HM^{-1}=P \diag \{ \lambda_1, 0,0  \} P^{-1}.
 $$
 Therefore 
 $$
 HM^{-1} -\omega^2= P \diag \{ \lambda_1 -\omega^2, -\omega^2,-\omega^2  \}P^{-1}.
 $$ 
 Thus the eigenvalues of $L$ restricted to $\tilde{E}$ are all purely complex if and only if $\lambda_1< \omega^2$. Straightforward computations lead to the value of $\lambda_1$. This remark completes the proof.
 \end{proof}
 
 It is interesting to compare the fixed points on $\mathbb S^1$ and  the collinear central configurations of the classical $n$-body problem. The fixed points of three particles are local maxima of the potential restricted on the equator, while the collinear central configurations are local minima of the potential restricted on a line. For each  collinear configuration of the  classical $n$-body problem, there is only one angular velocity to make a circular motion, while any angular velocity leads to circular motion in our case. And it is very interesting to notice that the stability depends on the velocity.\\
 
\noindent{\bf Acknowledgments.} This research was supported in part by an NSERC of Canada Discovery Grant (Florin Diacu), 
a CONACYT Fellowship (Juan Manuel S\'anchez-Cerritos), and a University of Victoria Scholarship (Shuqiang Zhu).


\end{document}